\documentclass[12pt]{amsart}
\usepackage{amsmath}
\usepackage{amssymb,latexsym}
\usepackage{mathrsfs}		
\usepackage[curve]{xypic}
\usepackage[dvips]{graphicx}
\usepackage{color}		
\usepackage{epsfig}
\usepackage{enumerate}
\usepackage[all]{xy}
\usepackage[colorlinks=true, urlcolor=blue,bookmarks=true,bookmarksopen=true, citecolor=blue,hypertex]{hyperref}
\usepackage{type1cm}
\numberwithin{equation}{section}

\newtheorem{definition}{Definition}[section]
\newtheorem{theorem}{Theorem}[section]

\newtheorem{corollary}[theorem]{Corollary}

\newtheorem{lemma}[theorem]{Lemma}
\newtheorem{proposition}[theorem]{Proposition}
\newtheorem{remark}[theorem]{Remark}

\def \bb{\mathbb}

\def \mc{\mathcal}
\def \mf{\mathfrak}

\def \RR{{\bb{R}}}

\def \({\left(}
\def \){\right)}
\def \<{\langle}
\def \>{\rangle}

\def \dsum{\oplus}

\def \union{\cup}

\def \vargeq{\geqslant}

\def \x{{\boldsymbol x}}
\def \w{{\boldsymbol w}}
\def \y{{\boldsymbol y}}
\def \q{{\boldsymbol q}}

\def \v{{\boldsymbol v}}
\def \e{{\boldsymbol e}}
\def \boldpi{{\boldsymbol \pi}}
\def \boldmu{{\boldsymbol \mu}}
\def \boldalpha{{\boldsymbol \alpha}}
\def \boldbeta{{\boldsymbol \beta}}
\def \B{{\boldsymbol B}}
\def \A{{\boldsymbol A}}
\def \E{{\boldsymbol E}}
\def \e{{\boldsymbol e}}
\def \Q{{\boldsymbol Q}}
\def \V{{\boldsymbol V}}
\def \R{{\mathbb R}}
\def \beq{\begin{equation}}
\def \eeq{\end{equation}}

\begin{document}
\title{Regularization of the Kepler problem on the Sphere}

\author{Shengda Hu}
\author{Manuele Santoprete }
\address{Department of Mathematics
Wilfrid Laurier University
75 University Avenue West,
Waterloo Ontario, Canada, N2L 3C5 }
 \email{shu@wlu.ca}
 \email{msantopr@wlu.ca}

\begin{abstract}
In this paper we regularize the  Kepler problem on $S^3$ in several
different ways. First, we perform a Moser-type regularization. Then, we
adapt the Ligon-Schaaf regularization to our problem. Finally, we show
that the Moser regularization and the Ligon-Schaaf map we obtained can be
understood as the composition of the corresponding maps for the Kepler problem
in Euclidean space and the  gnomonic transformation.
\end{abstract}
\maketitle

\section{Introduction}
It often happens that the flow associated to a vector field is incomplete. 
Famous examples are the Newtonian $n$-body problem,  where  singularities arise because  of the  existence of collision orbits, and the Kepler problem. 
 Regularization of vector fields is a common procedure in the study of differential equations. There are two main approaches. In the first approach  the  incompleteness of the flow is  removed by embedding it into a complete flow. The qualitative behavior of solutions off the set of singularities is the same for both vector fields. 
The second approach  involves surgery  and it is usually called regularization by surgery or block-regularization. Roughly, the idea is to excise a neighborhood of the singularity from the manifold on which the vector field is defined and then to identify appropriate points on the boundary of the region. 

Both approaches have been applied to the Kepler problem in Euclidean space.  The second approach was first used for the Kepler problem by Easton \cite{Easton}, while the first approach has several variants. We mention only the most relevant for our work. 
As far as we know the regularization of the planar Kepler  problem using the first approach was first discussed by Levi-Civita \cite{Levi-Civita}. Another beautiful  incarnation of the second approach, due  to Moser \cite{Moser}, consist in showing that the flow of the $n$-dimensional Kepler problem on surface of constant negative energy is conjugate to the geodesic flow on the unit tangent bundle of $S^n$. 
The main disadvantage of Levi-Civita and Moser's regularization methods is
that they  handle separately each energy level. This disadvantage is partially
removed by a regularization procedure due to Ligon and Schaaf \cite{Ligon}. 
The Ligon-Schaaf regularization  procedure allows to handle together all negative 
(resp., all positive) energy levels. However, negative, positive and zero energy levels still cannot be
 handled together with this procedure.
The treatment of the regularization in the original
 article by Ligon and Schaaf  requires laborious computations, and a somewhat simplified treatment of the Ligon-Schaaf
 regularization map is due to  Cushman and Duisteermaat \cite{CD} and Cushman and Bates   \cite{Cushman}.
Recently  Marle \cite{Marle}, and  Heckman and de Laat \cite{Heckman} posted on arXiv preprints that
give another simplified treatment by showing that the Ligon-Schaaf map can be  understood as an adaptation of the Moser
regularization map.
Regularization by surgery handles all the energy level at once, however this is a completely different approach that has other 
shortcomings. In fact, on the one hand it is difficult to find the attaching map and on the other hand this kind of regularization 
is not too helpful in understanding  the global flow of the system and the near-collision orbits.

The Kepler problem on the three-sphere $S^3$ is the main topic of this paper.
The Kepler problem  and $n$-body problem on spaces of constant curvature are over a century old problems and  recently
 have generated a good deal of scholarly interest. See \cite{Santoprete2012_1,Santoprete2012_2, Diacu_2012, Santoprete2008,
Santoprete2009} 
for some recent results and some history. See \cite{Diacu_2012} for a more exhaustive survey of recent results. 

The flow of the Kepler problem on the three-sphere $S^3$ is incomplete and 
several procedures have been used to regularize its vector field. The surgery approach to regularization was applied in \cite{Santoprete2009} to the Kepler problem on a  class of surfaces of revolution that include the two-sphere.
The Levi-Civita and Moser's approaches were used in \cite{Santander}. 

In the present article, we start by deriving the equations of motion
using the method of Dirac brackets (section 2). In section 3, we
describe the integrals of motion of the system and derive the equations
of the orbit. In section 4, we describe  the Moser regularization
applied to the Kepler problem on $S^3$. The details are given for the
negative energy level sets, where the orbits lie completely in the upper
hemisphere. The same approach applies to the part of any energy level
set that lies in the upper hemisphere. Furthermore, we carry out the
Ligon-Schaaf regularization directly (section 5) and compare it with the
case of Euclidean Kepler problem. It turns out that the Ligon-Schaaf
regularization for the negative energy part the two systems are
naturally related by the gnomonic map (section 6). The gnomonic map
takes the upper hemisphere to the full Euclidean space, and the two
Kepler systems are related by a rescaling of the induced map on the
tangent bundle. More precisely,  let $\Phi$ and $\Phi_c$ be the
Ligon-Schaaf maps for the Kepler problem on $S^3$ and $\RR^3$,
respectively, and let $\Psi$ denote the map induced by the gnomonic
projection, then $\Phi=\Phi_c\circ\Psi$. The  map $\Psi$ is not
symplectic, which implies the non-symplecticness of the Ligon-Shaaf map
for the spherical case. In the same section, we also show that the
gnomonic transformation relates the Moser regularization of the two
systems.

The relations uncovered in this paper are summarized in the diagram
below.

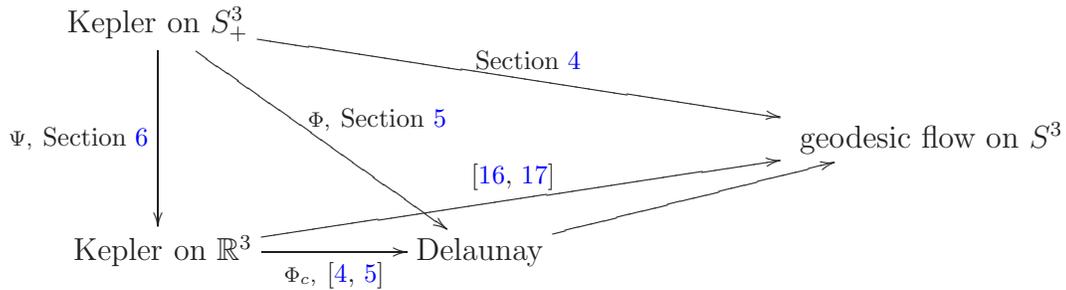
\[
\xymatrix{ \mbox{Kepler on } S_+^3 \ar[dd]_{\Psi, ~{\footnotesize \mbox{Section \ref{sec:gnomonic_transformation}}}} 
\ar[rrdd]^-{{\Phi,~ \footnotesize \mbox{Section \ref{sec:LS-regularization}}}}
 \ar[rrrrrd]^-{{\footnotesize \mbox{Section \ref{sec:Moser}}}} &
& 
\\ 
& & & & &\mbox{ geodesic flow on } S^3
\\                                                                    
\mbox{ Kepler on } {\mathbb R}^3 \ar[rr]_-{\Phi_c,~ {\footnotesize \mbox{\cite{CD,Cushman}}}}
 \ar[rrrrru]^-{{\footnotesize \mbox{\cite{Milnor,Moser}}}}&
& \mbox{Delaunay}  \ar[rrru]
}
\]


\section{Preliminaries}
Let $\langle\cdot,\cdot\rangle$ be the Euclidean inner product  in $\R^4$, and $\|\cdot\|$ the Euclidean norm in $\R^4$, and let $\cdot$, $\times$, and $|\cdot|$, be the usual dot product, vector product and Euclidean norm in $\R^3$, respectively .
Let $q=(q_0,\q)$ with $\q=(q_1,q_2,q_3)$ and $v=(v_0,\v)$ with $\v=(v_1,v_2,v_3)$ be canonical coordinates in $T\R^4$,  with the symplectic $2$-form $\displaystyle{\omega= \sum_idq_i\wedge dv_i}$.

Consider the following Hamiltonian system  $(H, T\R^4,\omega)$ with Hamiltonian 
\beq
H(q,v)=\frac 1 2\langle v,v \rangle +V(q) 
\eeq
where the potential energy $V(q)$ is
\beq
V(q)=-\gamma\frac{q_0}{(1-q_0^2)^{1/2}}
\eeq
This Hamiltonian describes  a particle that moves  in $\R^4$ under the influence  of potential $V(q)$.

When the particle is constrained to move on the unit $3$-sphere $S^3\subset \R^4$, the system is restricted to the tangent bundle of the 3-sphere:
\[
TS^3=\{(q,v)\in T\R^4|\langle q,q\rangle-1=0 \mbox{ and } \langle q, v\rangle =0\}
\]

The Hamiltonian vector field $X_H$ of $H$ on $T\RR^4$ does not restrict to the Hamiltonian vector field $X_{H|_{TS^3}}$ of the constrained system. There are two approaches in computing $X_{H|_{TS^3}}$ using either a modified Hamiltonian function $H^*$ (resulting in $X_{H^*}$, see Cushman \cite{Cushman}) or a modified Poisson bracket, the Dirac-Poisson bracket $\{\cdot, \cdot\}^*$ (resulting in $X_H^*$) on $T\RR^4$, such that the restriction of the resulting vector fields coincide with $X_{H|_{TS^3}}$.
The two approaches are related since
\[
\{F,G\}^*|_{TS^3}=\{F^*,G^*\}|_{TS^3}
\]
 %
The Dirac-Poisson bracket can be explicitly written down in this case.
\begin{lemma}\label{lemma:diracpoisson}
The Dirac-Poisson structure $\{\cdot,\cdot\}^*|_{TS^3}$ is 
\begin{center}
\begin{tabular}{c|cccccccc}
&$q_0$&$q_1$&$q_2$&$q_3$&$v_0$&$v_1$&$v_2$&$v_3$\\
\hline
$q_0$&0  &0 &0 &0 &$1-q_0^2$ &$-q_0q_1$ &$-q_0q_2$ &$-q_0q_3$\\
$q_1$&   &0 &0 &0 &$-q_0q_1$ &$1-q_1^2$ &$-q_1q_2$ &$-q_1q_3$\\
$q_2$&   &  &0 &0 &$-q_0q_2$& $-q_1q_2$&$1-q_2^2$ &$-q_2q_3$ \\
$q_3$&   &  &  &0 & $-q_0q_3$&$-q_1q_3$ &$-q_2q_3$ &$1-q_3^2$ \\
$v_0$&   &  &  &  &0 & $q_1v_0-q_0v_1$ &$q_2v_0-q_0v_2 $ &$q_3v_0-q_0v_3$\\
$v_1$&   &  &  &  & &0 &$ q_2v_1-q_1v_2$ &$q_3v_1-q_1v_3$\\
$v_2$&   &  &  &  & & &0 &$q_3v_2-q_2v_3$\\
$v_3$&   &  &  &  & & & &0\\
\end{tabular}
\end{center}
or 
\[
\{q_{\alpha},v_{\beta}\}^*|_{TS^3}=\delta_{\alpha\beta}-q_{\alpha}q_{\beta},\quad \{v_{\alpha},v_{\beta}\}^*|_{TS^3}=L_{\alpha\beta}, \quad \{q_{\alpha},q_{\beta}\}^*|_{TS^3}=0
\]
where $\alpha,\beta=0,1,2,3$ and $L_{\alpha\beta}=q_{\beta}v_{\alpha}-q_{\alpha}v_{\beta}$.
\end{lemma}
\begin{proof}
 The phase space $TS^3$ is given as a subset of $T\bb R^4$ by the constraints
\[
c_1(q,v)=\langle q, q\rangle-1=0\mbox{  and  } c_2(q,v)=\langle q,v\rangle =0
\]

The Dirac-Poisson brackets are given by the relation
\[
\{F,G\}^*=\{F,G\}+\sum_{i,j} C_{ij}\{F,c_i\}\{G,c_j\}
\]
where  $C_{ij}$ are the elements of the inverse of 
the matrix with entries $\{c_i,c_j\}$. In this case 
\[
C=\frac{1}{2\langle q,q\rangle }
\begin{bmatrix}
0&-1\\
1&0\\
\end{bmatrix}
\]
The result follows from a computation. 
\end{proof}

\begin{lemma}\label{lemma:eqmotion}
 The constrained equations on $TS^3$ are
\beq
\begin{split}
 \dot q&=v\\
\dot v&=-\nabla V(q)-(\langle v,v\rangle-\langle q,\nabla V(q)\rangle)~q\\
\end{split}
\eeq
\end{lemma}
\begin{proof}

We first compute the equations of motion for the Hamiltonian function $H$ with respect to the Dirac-Poisson bracket 
\[
\begin{split}
 \dot q&=\{q, H\}^*\\
\dot v&=\{v,H\}^*\\
\end{split}
\]
then we restrict the resulting vector field  $X^*_H$ to $TS^3$.

\end{proof}
Let $\{e_0,e_1,e_2,e_3,e_4\}$ be the standard basis for $\R^4$. Then for our choice of the potential, the equations  of motion take the form: 
\beq\label{eqmotion}
\begin{split}
 \dot q&=v\\
\dot v&=\frac{\gamma e_0}{(1-q_0^2)^{3/2}}-(\langle v,v\rangle+\gamma\frac{q_0}{(1-q_0^2)^{3/2}})~q\\
\end{split}
\eeq
\section{Conserved quantities and the equations of the trajectory}
It is easy to see that the Hamiltonian $H$ and the angular momentum $\boldmu = \q \times \v$ are integrals of motion of the vector field $X_{H^*}|_{TS^3}=X_{H|_{TS^3}}$.
%
There is another interesting integral
\[
\A=\boldpi \times \boldmu -\gamma \frac{\q}{|\q|}
\]
where $\boldpi=q_0\v-v_0\q$. The vector $\A$ is a spherical generalization of the vector known as 
the {\itshape Runge-Lenz vector} in  Euclidean space. Priority,  for the Runge-Lenz vector in Euclidean space,  is sometimes attributed to Laplace, 
but  appears to be  due to Jakob Hermann and Johann Bernoulli \cite{Goldstein1976}. 
The only difference in form between this integral and the one for the Kepler problem in Euclidean space is that in the latter problem the 
velocity $\v$ appears in $\A$ instead of $\boldpi$. 
 Note that, as in the case of the Kepler problem in Euclidean space, there is no universally accepted definition of the Runge-Lenz vector. The most common definition is given above, while the common alternative $\displaystyle{\e = \frac{\A}{\gamma}}$ is also called the {\itshape eccentricity vector}.
\begin{proposition}
The Runge-Lenz vector 
is an integral of motion.
\end{proposition}
\begin{proof}
\beq\begin{split}
\frac{d\A}{dt}=&\frac{d\boldpi}{dt}\times \boldmu-\gamma \frac{\dot \q}{|\q|}+\gamma\frac{ (\q\cdot\dot\q) \q}{|\q|^3}\\
              =&(\dot q_0\v+q_0\dot\v-\dot\q v_0-\q\dot v_0)\times \boldmu -\gamma \frac{\dot \q}{|\q|}+\gamma\frac{(\q\cdot\dot\q) \q}{|\q|^3}\\
              =&\left[v_0\v-\left(\langle v,v\rangle q_0+\gamma \frac{q_0^2}{|\q|^{3/2}}\right)\q-v_0\v-\q\left(\frac{\gamma}{| \q|^{3/2}}-\langle v,v\rangle q_0-\gamma\frac{q_0^2}{|\q|^{3/2}} \right)\right]\times \boldmu\\
               &-\gamma \frac{\v}{|\q|}+\gamma\frac{(\q\cdot\v) \q}{|\q|^3}, \mbox{  using (\ref{eqmotion})}\\
=&-\frac{\gamma}{|\q|^{3/2}}(\q\times\boldmu)-\frac{\gamma}{|\q|^3} (\v( \q\cdot\q)-\q( \q\cdot\v) )=\frac{\gamma}{|\q|^{3/2}}(\q\times \boldmu-\q\times \boldmu)=0
\end{split}
\eeq
\end{proof}
\begin{remark}
In analogy with the Euclidean case, there is an another conserved vector, the binormal vector,
\[
{\boldsymbol B}=\boldpi-\frac{\gamma}{|\q||\boldmu|^2}(\boldmu\times \q)
\]
It can be shown that $\A=\B\times \boldmu$, and $\B$ is ``binormal'' because it is normal to both $\A$ and $\boldmu$. 
\end{remark}

Regard $\q, \v, \boldmu,\A,\B$ as vectors in $\R^3$. Since $\boldmu=\q\times\v$ is a constant of motion, 
it points in a fixed direction. Moreover it is orthogonal to $\q$ and $\v$. Therefore,  when moving along an orbit,
$\q$ varies on a plane (in  $\R^3$) orthogonal to $\boldmu$.
Furthermore, it can be shown that  $(\A\times \B)\cdot \q=0$,  $\A\cdot \boldmu=0$, and $\B\cdot\boldmu=0$.
Hence, if $\A$ and $\B$ are both non-zero  they form a basis for the plane where $\q$ lies (however, they 
are zero on circular orbits). 

Hence, each orbit must lie on a three-dimensional subspace. Moreover, since $\langle q,q\rangle=1$, the orbit is a curve on the two-sphere $S^2$. 
If   $\phi$ is used  to denote the  angle between $\q$ and the fixed direction of $\A$ then
\beq\label{eqeq}
\A\cdot\q=|\A||\q|\cos\phi=\q\cdot(\boldpi\times\boldmu)-\gamma \frac{\q\cdot\q}{|\q|}
\eeq
By permutation of the terms in the scalar triple product and note that $\boldpi=q_0\v-v_0\q$,
\[
\q\cdot(\boldpi\times\boldmu)=\boldmu\cdot(\q\times \boldpi)=q_0\boldmu\cdot(\q\times\v)=q_0|\boldmu|^2
\]
Let $|\q|=r$, and $q_0=z$. Rearranging equation (\ref{eqeq}) yields
\beq\label{eqtrajectory}\begin{split}
&\frac z r=\frac{\gamma}{|\boldmu|^2}\left(1+\frac{|\A|}{\gamma}\cos\phi\right)\\
&r^2+z^2=1
\end{split}
\eeq
 where $(r,\phi,z)$ are cylindrical  coordinates on the space spanned by $A=(0,\A)$, $B=(0,\B)$ and $e_0$
 (where $\{e_0,e_1,e_2,e_3\}$ is the standard basis for $\R^4$).

Write the equation \eqref{eqtrajectory} in spherical coordinates $(\rho,\phi,\theta)$ using the formulas
$r=\rho\sin\theta$, $\phi=\phi$, and $z=\rho\cos\theta$, we obtain
\[
\frac{1}{\tan\theta}=\frac{\gamma}{|\boldmu|^2}\left(1+\frac{|\A|}{\gamma}\cos\phi\right)\\
\]
This recovers the known formula for the orbits of the Kepler problem on the sphere \cite{Carinena2005} (compare also with \cite{Santoprete2008}). The orbits are always closed. Each orbit is obtained by intersecting a conical surface with vertex at the origin  and a sphere, which reduces to a circle when $|\A|=0$. In fact, the first of equations (\ref{eqtrajectory}) is a conical surface (and, in particular, one nappe of a  quadric conical surface for $\displaystyle{\frac{|\A|}{\gamma}<1}$), since for $\phi=\phi_0$ it reduces to the equation of a line through the origin, and for $z=z_0$ it reduces to a conic section. 
The constant $\displaystyle{\epsilon =|\e|= \frac{|\A|}{\gamma}}$ is called the \emph{eccentricity} of the orbit, which explains why $\e$
 is called eccentricity vector. When $\epsilon=0$, the orbits are circles, 
when $0 < \epsilon < 1$ they are curves corresponding to ellipses on the plane, $\epsilon= 1$ corresponding to  parabolas, 
and $\epsilon > 1$ corresponding to hyperbolas. Note that the orbits are confined 
to the upper hemisphere when $0<\epsilon<1$, while it is not the case when $\epsilon\vargeq 1$.
 We refer the reader to  \cite{Carinena2005, Santoprete2008} for a more detailed discussion of the orbits. 

The Hamiltonian  $H$ can be rewritten as
\[
H=\frac 1 2(|\boldpi|^2+|\boldmu|^2)+V(q)
\]
and the length of the eccentricity $\e$ vector is 
\[
|\e|^2=1+\frac{|\boldmu|^2}{\gamma^2}\(2H-|\boldmu|^2 \)
\]
There is another constant of motion
\[
E= H-\frac{|\boldmu|^2}{2}= \frac 1 2|\boldpi|^2+V(q)
\]
which satisfies the following equation:
\beq\label{eq:e}
|\e|^2=1+2\frac{E}{\gamma^2}|\boldmu|^2
\eeq
Define the \emph{modified  eccentricity vector} $\widetilde \e=-\nu{\e}$, where 
$${\nu=\frac{\gamma}{\sqrt{-2 E}}}$$
then the following hold
\[\begin{split}
\boldmu\cdot \widetilde \e&=0\\
\|\boldmu\|^2+\|\widetilde \e\|^2&=\nu^2>0
\end{split}
\]
where the second equation follows from (\ref{eq:e}). These relations are equivalent to 
\[(\boldmu+\widetilde \e)\cdot(\boldmu +\widetilde \e)=(\boldmu-\widetilde \e)\cdot(\boldmu-\widetilde \e)=\nu^2>0\]
which defines a smooth 4-dimensional manifold diffeomorphic to $S^2_\nu\times S^2_\nu$.

\section{Hodograph and Moser's regularization}\label{sec:Moser}

In this section, we describe the Moser's regularization for the Kepler problem on $S^3$. The approach follows closely Milnor's work \cite{Milnor} for the Euclidean case.
Without loss of generality, suppose  that $\boldmu=(0,0,\mu)$. Then $\mu = |\boldmu|$, $\q=(q_1,q_2,0)$ and
 $\v=(v_1,v_2,0)$. Furthermore, suppose that $\A = (|\A|,0,0) \neq 0$ and  let $\phi$ be the angle formed by $\q$ with $\A$ (as in the equation of the trajectory). Since $\A=\B\times\boldmu$,
it follows that $\displaystyle{|\B|=\frac{|\A|}{|\boldmu|}(0,1,0)}$.
Then $\boldmu\times\q=|\boldmu||\q|(-\sin\phi,\cos \phi,0)$, and from the definition of $\B$, it follows
\[
\boldpi=\B+\frac{\gamma}{|\mu|}(\boldmu\times \q)=\frac{|\A|}{|\boldmu|}(0,1,0)+
\frac{\gamma}{|\boldmu|}(-\sin\phi,\cos\phi,0)=\frac{\gamma}{|\boldmu|}(-\sin\phi,\frac{|\A|}{\gamma}+\cos\phi,0)
\]
Thus, for each orbit, $\boldpi$ moves along  a circle centered at 
$\displaystyle{\frac{|\A|}{|\boldmu|}}$, with radius $\displaystyle{\frac{\gamma}{|\boldmu|}}$ and lying 
in the plane through the origin that is orthogonal to $\boldmu$. 

Conversely, starting with a circle with the equation 
$$\boldpi=\frac{\gamma}{|\boldmu|}\left(-\sin\phi, \frac{|\A|}{\gamma}+\cos\phi,0\right)$$
a direct computation gives 
$$\boldmu \cos \theta = q_0 \boldmu = \q \times \boldpi = \frac{\gamma}{|\boldmu|}\sin\theta\left(0, 0, 1 +  \frac{|\A|}{\gamma} \cos \phi\right)$$
Hence, we  recover the equation of the orbit  in spherical coordinates 
$$\frac{1}{\tan \theta} = \frac{\gamma}{|\boldmu|^2}\left(1 + \frac{|\A|}{\gamma} \cos \phi\right)$$
It follows that the hodocycle completely determines the corresponding Kepler orbit.
We can now prove the  following theorem.

\begin{theorem}
 Fixing some constant value of $E=H-|\boldmu|^2/2<0$, consider the space $M_{E}^+$ 
of all the vectors $\boldpi$ such that $\boldpi\cdot\boldpi>2E$, together with a single improper point $\boldpi=\infty$.
Such  space possesses one and only one Riemannian metric $ds^2$ so that the arc-length parameter $\int ds$ along any circle 
$t\rightarrow \boldpi(t)$ is precisely equal to the parameter $\displaystyle{\int \frac{dt}{q_0|\q|}}$. This  metric is smooth and complete, with constant curvature 
$-2E$, and its geodesics are precisely the circles or lines $t\rightarrow \boldpi(t)$ associated with Kepler orbits.
\end{theorem}

\begin{proof}
 From (\ref{eq:energy}) it is clear that if $E<0$ the orbits must have $q_0>0$ and thus are limited to the upper 
hemisphere. In this case $\displaystyle{\epsilon=\frac{|\A|}{\gamma}<1}$ and  only the space $M_{E}^+$ is relevant. 
In fact $\displaystyle{E=\frac 1 2 |\boldpi|^2-\gamma\frac{q_0}{(1-q_0^2)^{1/2}}}$ and thus, if $E<0$, $|\boldpi^2|>2E$, 
since in such case $q_0>0$. However, if $E>0$ things get more complicated.  

Using the equation of motions it is easy to show that 
 \[
\dot \boldpi=\dot q_0\v+q_o\dot \v-\dot v_0\q-v_0\dot\q=-\frac{\gamma \q}{|\q|^3}
\]
and thus $\displaystyle{|\dot \boldpi|=\frac{\gamma}{|\q|}}$. Dividing this equation by the definition of the time rescaling 
$\displaystyle{ds=\frac{\gamma}{q_0|\q|}dt}$ and using the fact that
\beq\frac{\gamma q_0}{|\q|}=\frac 1 2|\boldpi|^2-\frac 1 2(2H-|\boldmu|^2)\label{eq:energy}\eeq
we obtain
\[\left|\frac{d\boldpi}{ds}\right|=\frac{\gamma|q_0|}{|\q|}=\left|\frac 1 2|\boldpi|^2-E\right|\]

\beq\label{eq:metric}
ds^2=\frac{4d\boldpi\cdot d\boldpi}{(|\boldpi|^2-(2H-|\boldmu|^2))^2}=\frac{4d\boldpi\cdot d\boldpi}{(|\boldpi|^2-2E)^2}
\eeq

Thus, there is one and only one Riemannian metric on the spaces $M_{E}^+$ which satisfies our condition, and it is given by formula
(\ref{eq:metric}).

To describe what happens in a neighborhood of infinity, we work with the inverted velocity coordinate $\w=\frac{\boldpi}{|\boldpi|^2}$ (see, for example,  Milnor's paper \cite{Milnor} for a discussion of inversion).
Since the differential of $\w$ is $\displaystyle{d\w=\frac{(\boldpi\cdot\boldpi)d\boldpi-2(d\boldpi\cdot \boldpi)\boldpi}{(\boldpi\cdot\boldpi)^2}}$
we have $\displaystyle{d\w\cdot d\w=\frac{d\boldpi\cdot d\boldpi}{(\boldpi\cdot\boldpi)^2}}$.
Hence 
\beq\label{eq:metric_inverted}
ds^2=\frac{4d\w\cdot d\w}{(1-2E\w\cdot\w)^2}
\eeq
The metrics given in equation (\ref{eq:metric}) and (\ref{eq:metric_inverted}) have constant curvature $-2E$.
\end{proof}
Since $E < 0$, the corresponding metric space is a round $3$-sphere and $t \to \boldpi(t)$ is thus a geodesic on a round $3$-sphere.

\section{Ligon-Schaaf Regularization}\label{sec:LS-regularization}
\subsection{$\mf {so}(4)$ momentum map}
The components of the angular momentum and of an opportunely rescaled eccentricity vector form a Lie algebra under Poisson bracket which is isomorphic to $\mf {so}(4)$. This gives the momentum map of the Kepler problem on the sphere. 

Let $\mf g = \mf{so}(4)$. For a suitably chosen basis $\{X_1, X_2, X_3, Y_1, Y_2, Y_3\}$, the Lie bracket is
\begin{equation}
\label{eq:so4brackets}
[X_i, X_j] = [Y_i, Y_j] = \epsilon_{ijk} X_k \text{ and } [X_i, Y_j] = \epsilon_{ijk} Y_k
\end{equation}
The identification $\mf g \cong \mf{so}(3) \dsum \mf{so}(3)$ can be seen via substitution $\displaystyle{\frac{1}{2}(X_i + Y_i)}$ and $\displaystyle{\frac{1}{2}(X_i - Y_i)}$.
The basis can be thought of as coordinate functions on $\mf g^*$, then the Lie bracket defines a Poisson bracket $\{,\}_{\mf g}$ on $C^\infty(\mf g^*)$, which defines the Lie-Poisson structure on $\mf g^*$. 

We first write down the  brackets of the components of $\boldmu$ and $\A$:
\begin{lemma}\label{lem:muAbrackets}
The brackets of the components of $\boldmu$ and $\A$ are:
\beq\{\mu_i,\mu_j\}^*|_{TS^3}=\epsilon_{ijk}\mu_k,\quad \{\mu_i,A_j\}^*|_{TS^3}=\epsilon_{ijk}A_k, \quad \{A_i,A_j\}^*|_{TS^3}=-2( H-|\boldmu|^2)\epsilon_{ijk}\mu_k
\label{eq:poissonbrackets}\eeq
\end{lemma}
In the Appendix, we give a sketch of the proof of the Lemma above and of  the Proposition below.
\begin{proposition}\label{prop:nu}
Let $\mf g = \mf{so}(4)$, $C(H) = H + \sqrt{\gamma^2 + H^2}$ and
$$\eta(|\boldmu|^2, H) = \frac{-|\boldmu|^2+C(H)}{|\A|^2}$$
Then the following is the momentum map of the Kepler problem on the sphere 
$$\rho = (\boldmu, \eta(|\boldmu|^2, H)\A) : TS^3 \to \mf g^*$$
It is a Poisson map with respect to the bracket $\{,\}^*$ on $TS^3$ and the Lie-Poisson bracket on $\mf g^*$.
\end{proposition}

\subsection{Delaunay Vector field}
Let $(x, y)$ be the coordinates on $T\RR^4 \cong \RR^4 \dsum \RR^4$. Then $TS^3 \subset T\RR^4$ is given by $\<x, y\> = 0$.
Let $T^+S^3=\{(x,y)\in TS^3|y\neq 0\}$ be the tangent bundle of $S^3$ less its zero section and $\tilde \omega = \omega|_{T^+S^3}$ the restriction to $T^+S^3$ of the standard symplectic form $\omega$ on $T\RR^4$. Consider the Delaunay Hamiltonian on $T^+S^3$:
\beq
{\mathcal H}=-\frac 1 2\frac{\gamma^2}{\langle y,y\rangle}
\eeq
which resembles  the Kepler Hamiltonian written in Delaunay coordinates
(see for example \cite{Abraham}). It is clear that $\mc H$ is invariant under the standard action of $SO(4)$ on $T^+S^3$.

The integral curves of the Delaunay vector field $X_{\mathcal H}$ 
satisfy
\beq\begin{split}\label{eq:Delaunay_vectorfield}
\frac{dx}{dt}&=\frac{\gamma^2}{\langle y,y\rangle^2}y  \\
\frac{dy}{dt}&=-\frac{\gamma^2}{\langle y,y\rangle}x 
\end{split}
\eeq
It can be proved that the Delaunay vector field $X_{\mathcal H}$ is a time rescaling of the Hamiltonian vector field of the geodesic flow on the unit sphere. The space $\mf{so}(4)^*$ can be naturally identified with $\bigwedge^2 (\RR^4)^*$. Under this identification, the momentum mapping of the standard action of $SO(4)$ on $T^+S^3$ is:
$${\mathcal J}(x, y)=x\wedge y \in \bigwedge{}^2 (\RR^4)^* \cong \mf{so}(4)$$
A more detailed description of these facts can be found in \cite{Cushman}.

\subsection{Regularization}
All the orbits of the Kepler problem on the sphere that have $E<0$ can be regularized at once, with a slightly modified Ligon-Schaaf map. Let 
\[J:TS^3 \to \mf{so}(4) : (q,v)\mapsto\left(\boldmu,\widetilde \e\right)\]
We begin our search for a  Ligon-Schaaf map by noting that the image of 
${\mathcal J}$
is the same as the image of 
$J$ 
(see \cite{Cushman} for a detailed study of the Delaunay vector field and its momentum map). This suggests that the two maps are somewhat related, even though $J$ is not a momentum map.  Note that this situation differs from the case of the Kepler problem in $\R^3$, studied by Cushman \cite{Cushman}, where the maps that are related are the momentum maps. 
\begin{theorem}
 The smooth map $\Phi:(q,v)\mapsto(x,y)$ intertwines the momentum map ${\mathcal J}$ and the map $J$, that is  $\Phi^*{\mathcal J}=J$,
 if an only if
\beq\label{eq:Phi}
\Phi : (q,v) \mapsto (x,y)=\left(\alpha\sin\varphi+\beta\cos\varphi,\nu(-\alpha\cos\varphi+\beta\sin\varphi)\right)
\eeq
where 
\begin{align}
&\alpha=(\alpha_0,\boldalpha)=\left(\frac{1}{\nu q_0}\q \cdot \boldpi,\frac{\q}{|\q|}-\frac{\q \cdot \boldpi}{\gamma q_0}\boldpi\right) \label{eq:alpha}\\
&\beta=(\beta_0,\boldbeta)=\left(\frac{|\q|}{\gamma q_0}\boldpi\cdot\boldpi-1,\frac{|\q|}{\nu q_0}\boldpi\right)\label{eq:beta}
\end{align}
$\varphi$ is an arbitrary smooth real function and $\displaystyle{\nu = \frac{\gamma}{\sqrt{-2E}}}$.
\end{theorem}

\begin{proof}
Suppose $\Phi$ intertwines the momentum map ${\mathcal J}$ and the map $J$. Write $x=(x_0,\x)\in \R^3\times \R=\R^4$, $y=(y_0,\y)\in \R^3\times \R=\R^4$. Then $\Phi^*{\mathcal J}=J$ is equivalent to

\beq\label{eq:equalityofmaps}
\begin{split}
\x\times \y &=\q\times \v\\
\x_0\y-y_0\x &=\frac{1}{\sqrt{-2E}}\left(\gamma\frac{\q}{|\q|}-\boldpi\times(\q\times \v))   \right)=M\q+N\boldpi
\end{split}
\eeq
where
\beq
M=\frac{1}{\sqrt{-2E}}\left( \frac{\gamma}{|\q|}-\frac{\boldpi\cdot\boldpi}{q_0} \right) \quad \mbox{and}\quad N=\frac{1}{\sqrt{-2E}}\left(\frac{\boldpi\cdot \q}{q_0}\right)
\eeq
Suppose $\displaystyle{\q\times \v=\frac{1}{q_0}(\q\times\boldpi)\neq 0}$, then it follows that $\x$ and $\y$ lie on the same plane of $\q$ and $\boldpi$. Since $\q$ and $\boldpi$ are linearly independent, we obtain
\beq\label{eq:sys1}
\begin{split}
 \x&=a\q+b\boldpi\\
 \y&=c\q+d\boldpi\\
\end{split}
\eeq
Since $\displaystyle{\q\times \v=\frac{1}{q_0}(\q\times\boldpi)=\x\times \y=(ad-bc)(\q\times \boldpi)}$ we find $\displaystyle{ad-bc=\frac{1}{q_0}}$.
Substituting (\ref{eq:sys1}) into  (\ref{eq:equalityofmaps}) and using the linear independence of $\q$ and $\boldpi$ gives a set of linear equations for $\x_0$ and $\y_0$. Since $\displaystyle{ad-bc=\frac{1}{q_0}}$ these equations may be solved to give
\[\begin{split}
x_0&=(aN-bM)q_0\\
y_0&=(cN-dM)q_0
\end{split}
\]
Since $(x,y)\in T^+S^3$,
\beq\label{eq:sys3} 
\begin{split}
 1&=\x\cdot \x+x_0^2\\
 0&=\x\cdot\y+x_0y_0
\end{split}
\eeq
Substituting the expressions for $\x$ and $\y$ and the expressions for $x_0$ and $y_0$ above into (\ref{eq:sys3}) yields
\begin{align}
&1=a^2(\q\cdot \q)+\left(\frac{\boldpi\cdot \q}{\sqrt{-2E}}a+\frac{\gamma q_0}{|\q|\sqrt{-2E}}b\right)^2\label{eq:sys4a}\\
&0=ac\left(\q\cdot\q+\frac{(\boldpi\cdot\q)^2}{(-2E)}\right)+(ad+bc)\left(\gamma\frac{(\boldpi\cdot\q)q_0}{|\q|(-2E)}\right)
+bd\left(\gamma^2\frac{q_0^2}{|\q|^2(-2E)}\right)\label{eq:sys4b}
\end{align}
where we have used the identities
\beq\begin{split}
\q\cdot\q+N^2q_0^2=\q\cdot\q+\frac{(\boldpi\cdot\q)^2}{(-2E)}\\
(\q\cdot\boldpi)-MNq_0^2=\gamma\frac{(\boldpi\cdot\q)q_0}{|\q|(-2E)}\\
\boldpi\cdot\boldpi+M^2q_0^2=\gamma^2\frac{q_0^2}{|\q|^2(-2E)}
\end{split}\eeq
which follow from the definition of $M$ and $N$ and the identity
\[
\boldpi\cdot\boldpi-\gamma\frac{q_0}{|\q|}=2E+\gamma\frac{q_0}{|\q|}
\]
Multiplying  (\ref{eq:sys4a}) by $c$ and (\ref{eq:sys4b}) by $-a$ and adding the resulting equations gives
\[
c=-a\frac{\gamma(\boldpi\cdot \q)}{|\q|(-2E)}-b\frac{\gamma^2q_0}{|\q|^2(-2E)}
\]
Similarly multiplying (\ref{eq:sys4a}) by $d$ and (\ref{eq:sys4b}) by $-a$ yields
\[
d=a\left(\q\cdot \q+\frac{(\boldpi\cdot\q)^2}{(-2E)}\right)\frac{1}{q_0}+b\frac{\gamma (\boldpi\cdot \q)}{|\q|(-2E)}
\]

All solutions of (\ref{eq:sys4a}) are parametrized by
\beq
\begin{split}
&a=\frac{1}{|\q|}\sin\varphi\\
&b=\frac{\sqrt{-2E}|\q|}{\gamma q_0}\cos\varphi-\frac{\boldpi\cdot \q}{q_0\gamma}\sin\varphi
\end{split}
\eeq
where $\theta$ is an arbitrary function. 
Substituting the previous equations in the expressions for $c$ and $d$ gives
\beq
\begin{split}
&c=-\frac{\gamma}{|\q|\sqrt{-2E}}\cos\varphi\\
&d=\frac{|\q|}{q_0}\sin\varphi+\frac{(\boldpi\cdot \q)}{q_0\sqrt{-2E}}\cos\varphi
\end{split}
\eeq

Conversely, a calculation shows that a map $\Phi$ of the form (\ref{eq:Phi}) intertwines the momentum map ${\mathcal J}$
 and the map $J$.
\end{proof}

\begin{corollary}\label{co:E}
 $\Phi$ intertwines $E$ and the Delaunay Hamiltonian, that, is, $\Phi^*{\mathcal H}=E$.
\end{corollary}
\begin{proof}

\[
(\Phi^*{\mathcal H})(q,v)=-\frac 1 2\frac{\gamma^2}{\langle y,y\rangle}=-\frac 1 2\frac{(-2E)\gamma^2}{\gamma^2|-\alpha \cos\varphi+\beta\sin\varphi|^2}
=E(q,v)\]
since $\langle \alpha,\alpha\rangle=1$, $\langle\beta,\beta\rangle=1$, $\langle\alpha,\beta\rangle=0$, as it will be shown below. 
\begin{align*}
\langle \alpha,\alpha\rangle &=1-2\frac{(\boldpi\cdot\q)^2}{|\q|\gamma q_0}+\frac{(\boldpi\cdot\q)^2
(\boldpi\cdot\boldpi)}{\gamma^2q_0^2}+\frac{(\boldpi\cdot\q)^2}{\nu^2 q_0^2}\\
&=1+2\frac{(\boldpi\cdot\q)^2}{\gamma^2q_0^2}\left(\frac{\boldpi\cdot\boldpi}{2}-\frac{\gamma q_0}{|\q|}\right)+
\frac{(-2E)(\boldpi\cdot \q)^2}{\gamma^2 q_0^2}\\
&=1
\end{align*}
\begin{align*}
\langle\beta,\beta\rangle&=1+\frac{(\boldpi\cdot\boldpi)^2|\q|^2}{\gamma^2q_0^2}-2\frac{(\boldpi\cdot\boldpi)|\q|}{\gamma q_0}+\frac{|\q|^2(\boldpi\cdot\boldpi)}{\nu^2q_0^2}\\
&=1+2\frac{|\q|^2(\boldpi\cdot\boldpi)}{\gamma^2q_0^2}\left(\frac{\boldpi\cdot\boldpi}{2}-\frac{\gamma q_0}{|\q|}+\frac{(-2E)}{2}\right)\\
&=1
\end{align*}

\begin{align*}
\langle\alpha,\beta\rangle=&-\frac{\boldpi\cdot\q}{\nu q_0}+\frac{|\q|}{\nu\gamma q_0^2}(\boldpi\cdot\q)(\boldpi\cdot\boldpi)
\\
&+\frac{\boldpi\cdot\q}{\nu q_0}-\frac{|\q|}{\nu\gamma q_0^2}(\boldpi\cdot\q)(\boldpi\cdot\boldpi)=0
\end{align*}

\end{proof}
The next result is a computation   we will use to understand the relationship between the  Kepler  vector field $X_H$  and the Delaunay vector field $X_{\mathcal H}$.
\begin{lemma}\label{lemma:Liederivatives}
 The derivatives of $\alpha$ and $\beta$ along the  flow generated by the vector field $X_H$ are:
 \[
 \frac{d\alpha}{dt}=\frac{\sqrt{-2E}}{q_0|\q|}\beta\quad , \quad   
 \frac{d\beta}{dt}=-\frac{\sqrt{-2E}}{q_0|\q|}\alpha
 \]
\end{lemma}
\begin{proof}
 Let $t\to (q(t),v(t))$ be an integral curve of the Kepler vector field $X_H$. Let $\alpha=\alpha(q(t),v(t))$ and 
 $\beta=\beta(q(t),v(t))$, where $\alpha=(\alpha_0,\boldalpha)$ and $\beta=(\beta_0,\boldbeta)$ are given by (\ref{eq:alpha})
 and (\ref{eq:beta}), respectively. Recall that $\pi=q_0\v-v_0\q$ so that $\displaystyle{\v=\frac{1}{q_0}(\boldpi+v_0\q)}$. Moreover a  simple computation, shows that $\displaystyle{\frac{d\boldpi}{dt}=-\frac{\gamma \q}{|\q|^3}}$. With these equations in mind, together with the expressions for $X_H$, we compute
\begin{align*}
 \frac{d\alpha_0}{dt}&=\frac{1}{\nu q_0}\left(\dot\boldpi\cdot\q+\boldpi\cdot \v\right)- \frac{1}{\nu q_0^2}(\boldpi \cdot \q)\dot q_0\\
 &=\frac{1}{\nu q_0}\left(-\frac{\gamma}{|\q|}+\frac{\boldpi\cdot\boldpi}{q_0}\right)+\frac{1}{\nu q_0^2}(\boldpi\cdot\q)v_0-\frac{1}{\nu q_0^2}(\boldpi\cdot\q)v_0\\
 &=\frac{1}{\nu q_0|\q|}\left(\frac{(\boldpi\cdot\boldpi)}{q_0\gamma}-1\right)=\frac{1}{\nu q_0|\q|}\beta_0
\end{align*}
and 
\begin{align*}
\frac{d\boldalpha}{dt}&=\frac{\v}{|\q|}-\frac{(\q\cdot \v)}{|\q|^3}\q-\frac{(\dot \boldpi\cdot \q)}{\gamma q_0}\boldpi
-\frac{(\boldpi\cdot \dot \q)}{\gamma q_0}\boldpi-\frac{(\boldpi\cdot \q)}{\gamma q_0}\dot\boldpi -\frac{(\boldpi\cdot \q)\boldpi}{\gamma q_0^2}v_0\\
&=-\frac{2\boldpi}{\gamma q_0^2}\left(\frac{\boldpi\cdot\boldpi}{2}-\frac{\gamma q_0}{|\q|}\right)\\
&=\frac{\sqrt{-2E}}{q_0|\q|}\left( \frac{|\q|}{\nu q_0}\boldpi   \right)= \frac{\sqrt{-2E}}{q_0|\q|}\boldbeta
\end{align*}

Similarly, to  verify the expression for $\displaystyle{\frac{d\beta}{dt}}$ we compute 
\begin{align*}
 \frac{d\beta_0}{dt}&=2\frac{(\boldpi\cdot\dot \boldpi)|\q|}{\gamma q_0}-\frac{(\boldpi\cdot\boldpi)|\q|}{\gamma q_0^2}v_0
 +\frac{(\boldpi\cdot\boldpi)}{q_0\gamma}\frac{(\q\cdot\v)}{|\q|}\\
 &= -2\frac{(\boldpi\cdot\q)}{q_0|\q|^2}-\frac{(\boldpi\cdot\boldpi)}{\gamma q_0^2}|\q|v_0+\frac{(\boldpi\cdot\boldpi)}{\gamma q_0^2|\q|}(\q\cdot\boldpi)+\frac{(\boldpi\cdot\boldpi)}{q_0^2\gamma|\q|}(\q\cdot\q)\\
 &=2\frac{\boldpi\cdot \q}{q_0^2|\q|\gamma}\left(\frac{\boldpi\cdot\boldpi}{2}-\frac{\gamma q_0}{|\q|}\right)=-\frac{\sqrt{-2E}}{q_0|\q|}\left( \frac{\sqrt{-2E}(\boldpi\cdot\q)}{\gamma q_0}\right)\\
 &=-\frac{\sqrt{-2E}}{q_0|\q|}\alpha_0
 \end{align*}
 and
 \begin{align*}
 \frac{d\boldbeta}{dt}&=\frac{|\q|\dot\boldpi}{\nu q_0}-\frac{|\q|\boldpi}{\nu q_0^2}v_0
 +\frac{(\q\cdot\v)\boldpi}{\nu q_0|\q|}\\
 &=\sqrt{-2E}\left( \frac{|\q|\dot\boldpi}{\gamma q_0}-\frac{|\q|\boldpi}{\gamma q_0^2}v_0
 +\frac{(\q\cdot\boldpi)\boldpi}{\gamma q_0^2|\q|}+\frac{|\q|^2\boldpi}{\gamma q_0^2|\q|}v_0\right)\\
 &=-\frac{\sqrt{-2E}}{q_0|\q|}\left(\frac{\q}{|\q|}-\frac{(\q\cdot\boldpi)\boldpi}{\gamma q_0}\right)
 =-\frac{\sqrt{-2E}}{q_0|\q|}\boldalpha
 \end{align*}
\end{proof}

We recall some standard terminology that we  will use to describe the relation between the Kepler  vector field  and the Delaunay vector field.

\begin{definition}
Let $X$ and $Y$ be vector fields on the manifolds $M$ and $N$, and let  $g^X_t:M\to M$ and $g^Y_t:N\to N$ be the corresponding flows.
 We say that $X$ and $Y$ (and the corresponding flows) are  {\itshape $C^k$-conjugate} 
if there is a $C^k$-diffeomorphism $\Phi:M\to N$ such that 
\[
\Phi(g^X_t(x))=g^Y_{t}(\Phi(x))
\]
We say that $X$ and $Y$ (and the corresponding flows) are  {\itshape $C^k$-equivalent} 
if there is a $C^k$-diffeomorphism $\Phi:M\to N$ that maps the orbits of $g^X$ onto the orbits of $g^Y$ and preserves 
the direction of time.  That is, there is a family of monotone increasing diffeomorphisms $\tau_x:\R\to \R$ such that
\[
\Phi(g^X_t(x))=g^Y_{\tau_x(t)}(\Phi(x))
\]
\end{definition}
\begin{remark}
 
Note that differentiating $\Phi(g^X_t(x))=g^Y_{\tau(x,t)}(\Phi(x))$ with respect to $t$ yields
\beq\label{eq:equivalence}
\Phi_*X(x)=\left(\frac{\partial \tau_x}{\partial t}\right) Y(\Phi(x))
\eeq
Hence, if this equation is satisfied, under the conditions of the definition above, the two vector fields are $C^k$ equivalent.
\end{remark}
The next theorem finally describes the relationship between the Kepler and the Delaunay vector fields. 
\begin{theorem}
The Kepler  vector field $X_H$  and the Delaunay vector field $X_{\mathcal H}$ are  $C^\infty$-equivalent with maps 
\[\Phi(q,v)=(\alpha\sin\varphi+\beta\cos\varphi,\nu(-\alpha\cos\varphi+\beta\sin\varphi))
\]
and  
\[\tau_q=\int_{t_0}^t \frac{dt}{q_0^2(t)}\]
where $\displaystyle{\varphi=\frac{1}{\nu q_0}\q\cdot\boldpi (= \alpha_0)}$.
\end{theorem}
\begin{proof}
 Differentiating 
 \[
 (x(t),y(t))=\Phi(q(t),v(t))=\left(\alpha\sin\varphi+\beta\cos\varphi,\nu(-\alpha\cos\varphi+\beta\sin\varphi)\right)
 \]
with respect to $t$ along a trajectory $(q(t),v(t))$ of $X_H$ and using Lemma \ref{lemma:Liederivatives} yields
\begin{align*}
 \frac{dx}{dt}&= \frac{d\alpha}{dt}\sin\varphi+\alpha\cos\varphi\frac{d\varphi}{dt}+\frac{d\beta}{dt}\cos\varphi-\beta\sin\varphi\frac{d\theta}{dt}\\
&=\left(\frac{\sqrt{-2E}}{q_0|\q|}-\frac{d\varphi}{dt}\right)(\beta\sin\varphi-\alpha\cos\varphi)=
\frac{\sqrt{-2E}}{\gamma}\left(\frac{\sqrt{-2E}}{q_0|\q|}-\frac{d\varphi}{dt}\right)y\\
&=\frac{1}{\nu}\left(\frac{\gamma}{\nu q_0|\q|}-\frac{d\varphi}{dt}\right)y
\end{align*}
and
\begin{align*}
 \frac{dy}{dt}&=\nu \left(-\frac{d\alpha}{dt}\cos\varphi+\alpha\sin\varphi\frac{d\theta}{dt}+\frac{d\beta}{dt}\sin\varphi +\beta\cos\theta\frac{d\theta}{dt}\right)=\\
&=-\nu\left(\frac{\sqrt{-2E}}{q_0|\q|}-\frac{d\varphi}{dt}\right)(\alpha\sin\varphi+\beta\cos\varphi)\\
&=-\nu\left(\frac{\gamma}{\nu q_0|\q|}-\frac{d\varphi}{dt}\right)x
 \end{align*}
Taking $\displaystyle{\varphi=\frac{1}{\nu q_0}(\q\cdot \boldpi)}$ it implies that
\[\begin{split}
\frac{d\varphi}{dt}-\frac{\gamma}{\nu q_0|\q|}&=\frac 1 \nu\left(\frac{\v\cdot\boldpi}{q_0}+ \frac{\q\cdot\dot \boldpi}{q_0}-\frac{\q\cdot\boldpi}{q_0^2}\dot q_0\right) -\frac{\gamma}{\nu q_0|\q|}\\
&=\frac 1 \nu\left(\frac{\boldpi\cdot\boldpi}{q_0^2}+v_0\frac{\q\cdot \boldpi}{q_0^2} -\frac{\gamma}{q_0|\q|}-v_0\frac{\q\cdot \boldpi}{q_0^2}\right)-\frac{\gamma}{\nu q_0|\q|}\\
&=\frac{1}{\nu q_0^2}\left(\boldpi\cdot\boldpi-2\frac{\gamma q_0}{|\q|}\right)=-\frac{\gamma^2}{\nu^3 q_0^2}
\end{split}\]
and thus by Corollary \ref{co:E} we obtain
\[
\Phi_*(X_H)=\begin{pmatrix}
       \frac{\gamma^2}{q_0^2\langle y,y\rangle^2}y\\
       -\frac{\gamma^2}{q_0^2\langle y,y\rangle}x
\end{pmatrix}=\frac{1}{q_0^2}X_{\mathcal H}
\]


Let 
 \[\tau_q=\int_{t_0}^t \frac{dt}{q_0^2(t)}\] 
then  \eqref{eq:equivalence} holds and hence the two vector field are smoothly equivalent. 
\end{proof}

 Note that the vector field of  the Kepler problem in $\R^3$ is $C^k$-conjugate to the one of the Delaunay vector fields
(\cite{CD,Cushman}). And the 
diffeomorphism $\Phi_c$ intertwining the two vector fields is  a symplectomorphism.  
However,  the vector field $X_H$ of  the  Kepler problem in $S^3$ is $C^k$-equivalent but not  $C^k$-conjugate  to the Delaunay vector field $X_{\mathcal H}$. 


\section{Gnomonic Transformation}\label{sec:gnomonic_transformation}
The computations in the previous section are analogous to that of \cite{CD,Cushman}.
 However, there is a more direct link between the Kepler problem on $S^3$ and in Euclidean space.
 We first recall a classical result due to Appell and Serret \cite{Appell, Serret}, relating the Kepler problem on the upper 
hemisphere $S^3_+$  to that on $\RR^3$ via the gnomonic transformation.

%
%
On the phase space $T_0\R^3=(\R^3-\{0\})\times \R^3$ with coordinates $(\Q,\V)$ and symplectic form
 $\displaystyle{\omega_3=\sum_{i=1}^3 dQ_i\wedge dV_i}$ consider the 
Kepler Hamiltonian
\[
H_K(\Q,\V)=\frac 1 2 \V\cdot \V-\frac{\gamma}{|\Q|}
\]
where $\cdot$ is the Euclidean inner product on $\R^3$ and $|\Q|$ is the length of the vector $\Q$. The integral curves of the 
Hamiltonian vector field $X_{H_K}$ on $T_0\R^3$ satisfy the equations
\beq\label{eq:Kepler}
\begin{split}
& \dot \Q=\V\\
&\dot\V=-\gamma\frac{\Q}{|\Q|^3}
\end{split}
\eeq
which (for $\gamma>0$) describes the motion of a particle of mass $1$ about the origin under the influence of the Newtonian gravity. The momentum map of this system is 
\[J_K=\left(\Q\times \V,-\frac{\gamma}{\sqrt{-2H_K}}\left(\frac{1}{\gamma}\V\times(\Q\times \V)-\frac{\Q}{|\Q|}\right)\right)\]
This is a classical system and a detailed analysis can be found in \cite{Cushman}, or in any good mechanics book. 

As before, $S^3$ is the unit sphere in $\RR^4$. We now consider the gnomonic projection, which is the  projection  onto the tangent plane at the north pole from the center of the sphere: 
\[
q = (q_0, \q) \mapsto \Q :=\frac{\q}{\sqrt{1-|\q|^2}} = \frac{\q}{q_0}
\]
The induced map on the tangent space at $\q$ is given by
$$v = (v_0, \v) \mapsto \frac{\v}{\sqrt{1-|\q|^2}} + \frac{(\q \cdot \v)\q}{\left(1-|\q|^2\right)^{\frac{3}{2}}} = \frac{1}{q_0}\v + \frac{\q\cdot \v}{q_0^3}\q = \frac{1}{q_0}\v - \frac{v_0}{q_0^2} \q = \frac{1}{q_0^2}\boldpi$$

\begin{theorem}\label{th:gnomonic}
The Kepler  vector field $X_H$ on $S^3$ (restricted to the upper hemisphere) and the Kepler  vector field $X_{H_K}$ on $\R^3$ are  $C^\infty$-equivalent with maps 
\[\Psi : (q,v) \mapsto (\Q,\V)=\left(\frac{\q}{q_0}, \boldpi \right)
\]
and  
\[\tau_q=\int_{t_0}^t \frac{dt}{q_0^2(t)}\]
The map $\Psi$ together with $t \mapsto \tau_q$ is called the \emph{gnomonic transformation}.

\end{theorem}
\begin{proof}
 Differentiating $(\Q(t),\V(t))=\Psi(\q(t),\v(t))$ with respect to $t$ along a trajectory $(\q(t),\v(t))$ of $X_H$ 
and using Lemma \ref{lemma:Liederivatives} yields
\[
\frac{d\Q}{dt}=\frac{1}{q_0}\v+\frac{\q\cdot\v}{q_0^3}\q=\frac{1}{q_0^2}\V
\]
and
\[\begin{split}
\frac{d\V}{dt}&=q_0\frac{d\v}{dt}
+\frac{|\v|^2}{q_0}\q+\frac{(\q\cdot \frac{d\v}{dt})}{q_0}\q+\frac{(\q\cdot\v)^2}{q_0^3}\q\\
&=-\left(|\v|^2+\frac{(\q\cdot\v)^2}{q_0^2}+\gamma\frac{q_0}{(1-q_0^2)^{3/2}}\right)(q_0^2+(\q\cdot\q))\frac{\q}{q_0}+\left(|\v|^2+\frac{(\q\cdot\v)^2}{q_0^2}\right)\frac{\q}{q_0}\\
&=-\gamma\frac{\q}{|\q|^3}=-\gamma\frac{\Q}{|\Q|^3}(1+|\Q|^2)=-\frac{\gamma}{q_0}\frac{\Q}{|\Q|^3}
\end{split}
\]
where we used the fact that $\displaystyle{q_0^2v_0^2=(\q\cdot\v)^2}$.
Consequently we obtain
\[
\Psi_*(X_{H})=\begin{pmatrix}
\frac{\V}{q_0^2}\\
-\gamma\frac{\Q}{q_0^2|\Q|^3}
\end{pmatrix}=\frac{1}{q_0^2}X_{ H_K}
\]
Let 
 \[\tau_q=\int_{t_0}^t \frac{dt}{q_0^2(t)}\] 
then  \eqref{eq:equivalence} holds and hence the two vector field are smoothly equivalent. 
\end{proof}

 The map $\Psi$ of the theorem above has properties that are analogous to the map $\Phi$. In fact it intertwines $E$ and the Kepler  Hamiltonian in Euclidean space, and it also  intertwines  the momentum map $J_K$ of the Kepler problem in Euclidean space  with the map $J$.
We prove these properties below.

\begin{proposition}
Consider the  smooth map $\Psi:(\q,\v)\to(\Q,\V)$ 
\begin{enumerate}
\item  $\Psi$ intertwines $E$ and the Kepler  Hamiltonian in Euclidean space, that, is, $\Psi^*{H_K}=E$.
\item  $\Psi$ intertwines the momentum map ${ J_K}$ and the map $J$, that is  $\Psi^*{J_K}=J$.
\end{enumerate}
\end{proposition}
\begin{proof}
To prove the first part of the proposition we compute  $\Psi^*{H_K}$.
A simple computation yields
\[\begin{split}
(\Psi^*{H_K})(q,v)&=\frac 1 2\V\cdot\V-\frac{\gamma}{|\Q|}=\frac 1 2(q_0^2|\v|^2+2q_0^2v_0^2+|\q|^2v_0^2)-\frac{\gamma q_0}{\sqrt{1-q_0^2}}\\
&=\frac 1 2\boldpi\cdot\boldpi - \frac{\gamma q_0}{\sqrt{1-q_0^2}}=E(q,v)
\end{split}
\]
since
\[
\V=q_0\v+\frac{\q\cdot\v}{q_0}\q=q_0\v-\frac{q_0v_0}{q_0}\q=\boldpi
\]
and
\[
\frac{\gamma}{|\Q|}=\gamma\frac{q_0}{|\q|}=\gamma\frac{q_0}{\sqrt{1-q_0^2}}
\]
We prove the second part by computing $\Psi^*{J_K}$.
A computation gives
\[
 (\Psi^*J_K)(q,v)=\left(\Q\times \V,-\frac{\gamma}{\sqrt{-2H_K}}\left(\frac{1}{\gamma}\V\times(\Q\times \V)-\frac{\Q}{|\Q|}\right)\right)=(\boldmu,\widetilde\e)=J(q,v)
 \]
 since
 \[
 \Q\times\V=\q\times\left(\v+\frac{(\q\times \v)}{1-|\q|^2}\q\right)=\q\times\v=\boldmu
 \]
 \[
 -\frac{\Q}{|\Q|}+\frac{1}{\gamma}\V\times(\Q\times\V)=-\frac{\q}{\sqrt{1-q_0^2}}+\frac 1 \gamma \boldpi\times\boldmu=\e
 \]
 and $\Psi^*{H_K}=E$ by part (1) of the proposition. 
\end{proof}

Recall that if $(M,\omega_M)$ and $(N, \omega_N)$ are symplectic manifolds and $f:M\to N$ is a diffeomorphism then $f$ is symplectic 
if and only if for all $h$, 
\beq\label{eq:symplecticmap}
f^*X_h=X_{h\circ f}
\eeq
see \cite{Abraham} for a proof.
With this in mind we show that $\Psi$ is not symplectic.
\begin{proposition}\label{prop:notsymplectic}
The map $\Psi$ is not symplectic. 
\end{proposition}
\begin{proof}
To show that the map is not symplectic it is enough to find an Hamiltonian for which  \eqref{eq:symplecticmap} is not satisfied.
Let $h=H_K$, and let $f=\Psi$, then it can be shown that  \eqref{eq:symplecticmap} is equivalent to writing
\[
X_{H_K}(\Psi(q,v))=\Psi_*X_{H_K\circ\Psi}(q,v)=\Psi_*X_{E}(q,v)
\]
The left-hand side of the equation is 
\beq
(X_{H_K}(\Psi(q,v))=\left(
\begin{array}{c}
\V\\
-\gamma\frac{\Q}{|\Q|^3}
\end{array}\right)=
\left(
\begin{array}{c}
q_0\v+(\q\cdot\v)q_0^{-1}\q\\
-\gamma\frac{\q q_0^2}{|\q|^3}
\end{array}\right)
\eeq

To compute the right hand side note that $\{\mu_i,q_j\}^*|_{TS^3}=\epsilon_{ijk}q_k$ and $\quad \{\mu_i,v_j\}^*|_{TS^3}=\epsilon_{ijk}v_k$ (see the Appendix), and thus $\{\q,|\boldmu|^2\}^*|_{TS^3}=\boldmu\times \q$ and 
$\{\v,|\boldmu|^2\}^*|_{TS^3}=\boldmu\times \v$. Consequently
\[
X_E^{(\q,\v)}=\left(
\begin{array}{c}
\{\q,H-|\boldmu|^2/2\}^*|_{TS^3}\\
\{\v,H-|\boldmu|^2/2\}^*|_{TS^3}
\end{array}\right)=
X_H^{(\q,\v)}-\frac 1 2\left(
\begin{array}{c}
\boldmu\times \q\\
\boldmu\times \v
\end{array}\right)
\]
where $X_E^{(\q,\v)}$ denotes the part of the vector field corresponding to $\displaystyle{\left.\left\{\q,H-\frac{1}{2}|\boldmu|^2\right\}^*\right|_{TS^3}}$ and
$\displaystyle{\left.\left\{\v,H-\frac{1}{2}|\boldmu|^2\right\}^*\right|_{TS^3}}$.
Differentiating $\Q(t)$ with respect to $t$ along a trajectory of $X_E$  yields
\[
\frac{d\Q}{dt}=\frac{\v-\frac 1 2\boldmu\times \q}{\sqrt{1-|\q|^2}}+\frac{\q[\q\cdot(\v-\frac 1 2\boldmu\times \q)]}{(1-|\q|^2)^{3/2}}
=\frac{1}{q_0^2}(q_0\v+(\q\cdot\v)q_0^{-1}\q)+\frac{1}{2q_0}\left(-|\q|^2\v+(\q\cdot\v)\q\right)
\]
This gives the first three components of $\Psi_*X_E$. Comparing with the first three components of $X_{H_K}(\Psi(q,v))$ yields

\[
-\frac{|\q|^2}{q_0^2}\left(q_0\v+(\q\cdot\v)q_0^{-1}\q\right)=\frac{1}{2q_0}\left(-|\q|^2\v+(\q\cdot\v)\q\right)
\]
If $\v$ is non-zero one can compare the coefficients of $\v$. This comparison yields the impossibility $\displaystyle{1=\frac{1}{2}}$, and hence 
the identity \eqref{eq:symplecticmap} is not satisfied for $h=H_K$.
\end{proof}

\begin{remark}
The non-symplecticness of the map $\Psi$ can be seen alternatively as follows. 
It can be shown that $\Psi$ is the composition of the tangent lift with a scaling $\kappa : TS^3_+ \to TS^3_+$ of the fiber direction by the factor of $q_0^2 = 1-|\q|^2$. As shown below, $\kappa$ is not symplectic with respect to the standard symplectic structure, which implies that the map $\Psi$ is not symplectic either:
$$\kappa^* d\q\wedge d\v = d\q \wedge d((1-|\q|^2)\v) = (1-|\q|^2)d\q \wedge d\v - 2\sum_{i, j = 1}^3 v_iq_jdq_i\wedge dq_j \neq d\q\wedge d\v$$
Here, $\q$ is used as coordinates on $S^3_+$, via the projection along the $q_0$ direction.
\end{remark}

\subsection{Relation to the Ligon-Schaaf regularization for the Kepler problem in $\RR^3$}

Recall the description of the Ligon-Schaaf regularization for the Kepler problem on $\RR^3$ as given in \cite{CD,Cushman}. 
The symplectomorphism $\Phi_c$ intertwining the vector fields of the Kepler problem and one of the Delaunay vector fields is given by
$$ \Phi_c: (\Q, \V) \mapsto (x_c, y_c) = (\alpha_c \sin \varphi_c + \beta_c \cos \varphi_c, \nu_c(-\alpha_c\cos \varphi_c + \beta_c\sin\varphi_c))$$
where $\displaystyle{\nu_c = \frac{\gamma}{\sqrt{-2H_K}}}$, $\varphi_c = \alpha_{c, 0}$ and
$$\alpha_c = (\alpha_{c, 0}, \boldalpha_c) = \left(\frac{1}{\nu_c}\V\cdot \Q, \frac{\Q}{|\Q|} - \frac{\V\cdot \Q}{\gamma}\V\right)$$
$$\beta_c =  (\beta_{c, 0}, \boldbeta_c) = \left(\frac{|\Q|}{\gamma}(\V\cdot\V) - 1, \frac{|\Q|}{\nu_c} \V\right)$$
This calculation proves the following 
\begin{proposition}
$\Phi_c \circ \Psi = \Phi$ 
\end{proposition}
A straightforward consequence of the proposition above is the following
\begin{corollary}
$\Phi$ is not symplectic. 
\begin{proof}
Suppose $\Phi$ is symplectic. Then we can rewrite the proposition above as $\Psi = \Phi_c^{-1}\circ\Phi$. Since  $\Phi_c^{-1}$ and $\Phi$ are symplectic
it follows that $\Psi$ is symplectic. This is in  contradiction with Proposition \ref{prop:notsymplectic}.
\end{proof}
\end{corollary}

\subsection{Relation to Moser's regularization for the Kepler problem in $\RR^3$}
First we recall Theorem $2$ in \cite{Milnor} for the Kepler problem in $\RR^3$. For a given constant value of $h = H_K$, 
define the space $M_h = \{\V| \V \cdot \V > 2h\}\union \{\infty\}$.
 Define also a Riemannian metric $\displaystyle{ds^2 = \frac{4d\V \cdot d\V}{(\V\cdot \V - 2h)^2}}$ on $M_h$. 
The arc-length parameter $\displaystyle{\int ds}$ of this metric along any velocity circle $\tau \to \V(\tau)$ is equal
to the parameter $\displaystyle{\int \frac{d\tau}{|\Q|}}$, where $\tau$ denotes the time. 

\begin{proposition}
The Moser's regularization for the Kepler problems in $\RR^3$ and the upper hemisphere are related by the gnomonic transformation.
\begin{proof}
The gnomonic transformation is given by $\displaystyle{(\Q, \V) = \Psi(q, v) = \left(\frac{\q}{q_0}, \boldpi\right)}$ and $\displaystyle{d\tau =\frac{dt}{q_0^2}}$ along the trajectory $q(t)$. It's clear that $\Psi$ maps $M_E^+$ to $M_E$ and the metrics correspond. The arc-length parameter becomes $\displaystyle{\int \frac{d\tau}{|\Q|} = \int\frac{\frac{dt}{q_0^2}}{\left|\frac{\q}{q_0}\right|} = \int \frac{dt}{q_0|\q|}}$.
\end{proof}
\end{proposition}


\section*{Appendix}

\begin{proof}[Proof of Lemma \ref{lem:muAbrackets}]
We verify only the third equation in (\ref{eq:poissonbrackets}). Using Lemma \ref{lemma:diracpoisson} we obtain
\[
\{\mu_i,q_j\}^*|_{TS^3}=\epsilon_{ijk}q_k, \quad \{\mu_i,v_j\}^*|_{TS^3}=\epsilon_{ijk}v_k,\quad \{\mu_i,\pi_j\}^*|_{TS^3} = \epsilon_{ijk}\pi_k
\]
and 

\[
\{\pi_i,\pi_j\}^*|_{TS^3}=q_iv_j-q_jv_i, \quad \{\mu_i,|\q|^2\}^*|_{TS^3}=0, \quad \left.\left\{\pi_i,\frac{1}{|\q|}\right\}^*\right|_{TS^3}=\frac{q_0q_i}{|\q|^3},\quad \{\pi_i, q_j\}^*|_{TS^3}=0
\]

Using the bilinearity property of Poisson brackets, expand 
\[
\{A_i,A_j\}=\left\{\epsilon_{lmi}\pi_l\mu_m-\gamma\frac{q_i}{|\q|},\epsilon_{pqj}\pi_p\mu_q-\gamma\frac{q_j}{|\q|}\right\}
\] 
to obtain 
\[
\quad \{A_i,A_j\}^*|_{TS^3}=-2( H-|\boldmu|^2)\epsilon_{ijk}\mu_k
\]
It helps to recall the identity
\[
\epsilon_{ijk}\epsilon_{ilm}=\delta_{km}\delta_{jl}-\delta_{jm}\delta_{kl}
\]
\end{proof}

\begin{proof}[Proof of Proposition \ref{prop:nu}.]
Let $\widetilde \E = \eta(|\boldmu|^2, H)\A$, and let $\widetilde \E=(\widetilde E_1, \widetilde E_2, \widetilde E_3)$, then the proposition amounts to showing that $\{\mu_1, \mu_2, \mu_3, \widetilde E_1, \widetilde E_2, \widetilde E_3\}$ satisfies \eqref{eq:so4brackets} with the Poisson bracket $\{,\}^*$.

We start with \eqref{eq:poissonbrackets} and
compute the brackets of the components of $\widetilde \E$. For $\eta = \eta(|\boldmu|^2, H)$,
$$\{\eta A_i,\eta A_j\}^*|_{TS^3}=\eta^2\{A_i,A_j\}^*|_{TS^3}+\eta\{A_i,\eta\}^*|_{TS^3} A_j+\eta\{\eta,A_j\}^*|_{TS^3} A_i$$

\[\begin{split}\{\eta, A_l\}^*|_{TS^3}&=\epsilon_{ijk}A_k\frac{\partial \eta}{\partial \mu_i}\delta_{jl}=\epsilon_{ilk}A_k\frac{\partial \eta}{\partial|\boldmu|^2}\frac{\partial |\boldmu|^2}{\partial \mu_i}=-2\epsilon_{ikl}\mu_iA_k\frac{\partial \eta}{\partial |\boldmu|^2}\\
   &=-2\frac{\partial \eta}{\partial |\boldmu|^2}(\boldmu\times \A)_l
\end{split}\]
Consequently 
\[
\{\eta A_i,\eta A_j\}^*|_{TS^3}=\eta^2\{A_i,A_j\}^*|_{TS^3}+2\eta\frac{\partial \eta}{\partial |\boldmu|^2}\left((\boldmu\times \A)_iA_j- (\boldmu\times \A)_jA_i\right).
\]
Simple computations show that
\[\begin{split}
\{\eta A_i,\eta A_j\}^*|_{TS^3}&=-2\eta^2(H-|\boldmu|^2)\epsilon_{ijk}\mu_k-\frac{\partial \eta^2}{\partial |\boldmu|^2}|\A|^2\epsilon_{ijk}\mu_k\\
                     &=\left[ -2\eta^2(H-|\boldmu|^2)-\frac{\partial \eta^2}{\partial |\boldmu|^2}(\gamma^2+2H|\boldmu|^2-(|\boldmu|^2)^2)\right]\epsilon_{ijk}\mu_k
\end{split}\]
Since $|\A|^2=\gamma^2+|\boldmu|^2(2H-|\boldmu|^2)$, we have that $\frac{\partial |\A|^2}{\partial |\boldmu|^2}=2(H-|\mu|^2)$. Therefore 
\begin{equation}
\label{eq:etildebracket}\{\widetilde E_i, \widetilde E_j\}^*|_{TS^3} = \{\eta A_i,\eta A_j\}^*|_{TS^3}=-\frac{\partial |\eta\A|^2}{\partial |\boldmu|^2}\epsilon_{ijk}\mu_k = -(\epsilon_{ijk}\mu_k)\frac{\partial}{\partial |\boldmu |^2} |\widetilde\e|^2
\end{equation}

Lastly, we determine the function $\eta$. The equation \eqref{eq:etildebracket} implies that $\eta$ has to satisfy
\[\frac{\partial |\eta\A|^2}{\partial |\boldmu|^2}=-1\]
Integrating we obtain $\eta^2|\A|^2=-|\boldmu|^2+C(H)$, and if $|\A|^2\neq 0$ we have 
\[\eta^2=\frac{-|\boldmu|^2+C(H)}{|\A|^2}\]
where $C(H)$ is an arbitrary function of $H$. For  $ {|\A|^2}=0$ we have $|\boldmu|^2=C(H)$. Substituting in
the expression $|\A|^2=\gamma^2+|\boldmu|^2(2H-|\boldmu|^2)$ yields 
\[
\gamma^2+C(H)(2H-C(H))=0
\]
then the positive solution gives $C(H) = H + \sqrt{\gamma^2 + H^2}$.
\end{proof}

\section*{Acknowledgements}
The research of S.H.  was supported by an  NSERC  Discovery Grant and  a Wilfrid Laurier start-up grant.  The research of M.S. was supported by  NSERC through a  Discovery Grant. 

\end{document}